\documentclass[12pt,reqno]{amsart}
\usepackage[english]{babel}
\usepackage[cp1251]{inputenc}

\usepackage{mathrsfs}
\usepackage{amsthm}
\usepackage{amssymb}
\usepackage{amsmath,amsthm,amsfonts}
\usepackage{mathabx}
\usepackage{hhline}
\usepackage{textcomp}
\usepackage{amsmath,amscd}
\usepackage[all,cmtip]{xy}
\usepackage{mathtools}
\usepackage[all]{xy}
\usepackage[margin=1in]{geometry}

\author{Kaveh Eftekharinasab}
\title{Sard's theorem for mapping between Fr\'{e}chet manifolds}
\address{Topology dept. \\ Institute of Mathematics of NAS of Ukraine \\ Te\-re\-shchen\-kivska st. 3, Kyiv, 01601 Ukraine}
\email{kaveh@imath.kiev.ua}

\makeatletter
\@addtoreset{equation}{section}
\@addtoreset{figure}{section}
\@addtoreset{table}{section}
\makeatother

\newtheorem{theorem}{Theorem}[section]
\newtheorem{lemma}{Lemma}[section]
\newtheorem{remk}{Remark}[section]
\newtheorem{prop}{Proposition}[section]
\newtheorem{defn}{Definition}[section]

\theoremstyle{definition}

\DeclareMathAlphabet{\mathpzc}{OT1}{pzc}{m}{it}

\DeclareMathOperator{\codim}{codim}
\DeclareMathOperator{\Ind}{Ind}
\DeclareMathOperator{\Img}{Img}

\DeclareMathOperator{\Aut}{Aut}
\DeclareMathOperator{\Iso}{Iso}
\begin{document}

\begin{abstract}
In this paper we prove an infinite-dimensional version of Sard's theorem for Fr\'{e}chet manifolds.
 Let $ M $ and $ N $ be bounded Fr\'{e}chet manifolds such that the topologies of their model Fr\'{e}chet spaces are defined by metrics with absolutely convex balls. 
 Let $ f: M \rightarrow N $ be an $ MC^k$-Lipschitz-Fredholm map with $  k > \max \lbrace {\Ind f,0} \rbrace $. Then the set of regular 
values of $ f $ is residual in $ N $.
\end{abstract}
\maketitle

\section{Introduction}
Sard's theorem in infinite-dimensional spaces may fail as showed in \cite{kupka} by giving a counterexample of
 real smooth map on a Hilbert space with critical values containing open set. However,  Smale \cite{smale}
 proved that if $ f : M \rightarrow N $ is a $ C^k $-Fredholm map between Banach manifolds with 
$ k > \max \lbrace \Ind f, 0\rbrace $, then the set of regular values of $ f $ is residual in $ N $. The 
assumption that $f$ must be Fredholm is necessary from the counterexample in \cite{kupka}.
In this paper we generalize the Smale's theorem to the Fr\'{e}chet manifolds case. To carry out this, at first we 
need to establish the stability of Fredholm operators under small perturbation which requires an 
appropriate topology on the space of linear continuous maps. But the space of continuous linear mappings of one Fr\'{e}chet space to another is not a Fr\'{e}chet space in general.
On the other hand, the general linear group of a Fr\'{e}chet space does not admit any non-trivial topological group structure. These defects put in question 
the way of obtaining  the stability of Fredholm operators.

There is a way to overcome aforementioned problems. Recently, in the suggestive paper~\cite{muller}, M\"{u}ller introduced the concept of bounded Fr\'{e}chet manifolds
and provided an inverse function theorem in the sense of Nash and Moser in this category. Such spaces arise in geometry and physical field theory and have
many desirable properties. For instance, the space of all smooth sections of a fiber bundle (over closed or non-compact manifolds), which
is the foremost example of infinite dimensional manifolds, has the structure of a bounded Fr\'{e}chet manifold, see~\cite[Theorem 3.34]{muller}. 
One of the essential ideas of this setting is to replace the space of all  continuous linear maps by the space $\mathcal{L}_{d',d}(E,F)$ of all linear Lipschitz continuous maps.
Then  $\mathcal{L}_{d',d}(E,F)$ is a topological group that has satisfactory properties. For example, the composition map 
$\mathcal{L}_{d,g}(F,G) \times \mathcal{L}_{d',d}(E,F) \rightarrow \mathcal{L}_{d',g}(E,G) $ is bilinear continuous. In particular,
the evaluation map $\mathcal{L}_{g,d}(E,F) \times E \rightarrow F$ is continuous. 

In this paper we will consider this category of manifolds. We introduce the concept of Lipschitz-Fredholm operator 
and prove  the stability theorem in Section~\ref{lp}. We then formulate Sard's theorem in Section~\ref{sar}. The key ingredient
to the proof of the theorem is the Local representation theorem (Theorem \ref{rp}) which is a consequence of the inverse function theorem. 
\section{Preliminaries and Notations}
In this section we set up our notations.
Most of the terminologies are taken from \cite{glockner}, but we avoid differing metric Fr\'{e}chet space with Fr\'{e}chet space.

\subsection{Lipschitz maps and the space $\mathcal{L}_{d,g}(E,F)$}
We denote by $(F,d)$ a Fr\'{e}chet space  whose 
topology is defined by a complete translational-invariant metric $d$. 
We define $  \parallel f \parallel_d \coloneq d(f,0) $ for $ f \in F $ and  write $ L.f $ instead of $ L(f) $ when $ L $ is a  linear map
between Fr\'{e}chet spaces. A metric with absolutely convex balls will be called a standard metric. Note that
every Fr\'{e}chet space  admits a standard metric which defines its topology: If $\alpha_n$ is an arbitrary sequence of 
positive real numbers converging to $zero$ and if $\rho_n$ is any sequence of continuous semi-norms defining the 
 topology of $F$. Then
\begin{equation*}
d_{\alpha,\, \rho}(e,f)\coloneq \sup_{n \in \mathbb{N}} \alpha _n \dfrac{\rho_n(e-f)}{1+\rho_n(e-f)}
\end{equation*}
is a metric on $F$ with the desired properties.
\begin{defn}
Suppose $(E,d)$ and $(F,g)$ are two Fr\'{e}chet spaces. Define $\mathcal{L}_{d,g}(E,F)$ to be the set of all globally Lipschitz linear maps, i.e., maps $ L: E \rightarrow F $ such that for them:
\begin{equation*}
\parallel L \parallel_{d,g}\, = \displaystyle \sup_{x \in E\setminus\{0\}} \dfrac{\parallel L.x \parallel_g}{\parallel x \parallel_d} < \infty.
\end{equation*}
We abbreviate $\mathcal{L}_{d}(E)=\mathcal{L}_{d,d}(E,E)$ and write $\parallel L\parallel_{d}\, = \,\parallel L\parallel_{d,d}$ for $L \in \mathcal{L}_{d}(E) $.
\end{defn}
\begin{remk}[\cite{glockner}, Remark 1.9] \label{p1}
 $\mathcal{L}_{d,g}(E,F)$ and the functions $\parallel .\parallel_{d,g}$ have the following useful properties:
\begin{description}
\item [i] $ \parallel L.x \parallel_g \,\leq \,\parallel L \parallel_{d,g} \,\parallel x \parallel_d $ for all $ x \in E $. Moreover, $ 0 \in \mathcal{L}_{d,g}(E,F)$ with $ \Vert \, 0 \, \Vert_{d,g} = 0 $. If $ L $ is not identically zero, then $ \parallel L_{d,g}  \parallel > 0 $. 
\item[ii] If $(G,h)$ is another Fr\'{e}chet space, then  
\begin{equation*}
\Vert H \circ L\parallel_{d,h} \, \leq \, \parallel H \parallel_{g,h} \, \parallel L\parallel_{d,g} ,
\end{equation*}
for $ L \in \mathcal{L}_{d,g}(E,F)$ and $ H \in \mathcal{L}_{g,h}(F,G)$.
\item[iii] If $ L,H \in \mathcal{L}_{d,g}(E,F) $, then
\begin{equation*}
\parallel L+H \parallel_{d,g} \, \leq \, \parallel L \parallel_{d,g} + \parallel H \parallel_{d,g} < \infty.
\end{equation*}
\item[iv] If g  is a standard metric, then
\begin{equation} \label{metric} 
 \begin{array}{cccc}
  D_{d,g}: \mathcal{L}_{d,g}(E,F) \times \mathcal{L}_{d,g}(E,F) \longrightarrow [0,\infty) , \\
(L,H) \longrightarrow \parallel L-H \parallel_{d,g}
 \end{array}
\end{equation}
is a translational-invariant metric on $\mathcal{L}_{d,g}(E,F)$ making it into an abelian topological group. 
\end{description}
\end{remk}
\begin{prop} [\cite {glockner}, Proposition 2.1]\label{p2}
Let $ (E,d) $ and $ (F,g) $ be Fr\'{e}chet spaces, and $ g $ a standard metric. Then the following hold:
\begin{description}
\item[i]$\mathcal{L}_{d,g}(E,F)$ is a vector subspace of the space of all 
maps from $ E $ to $ F $.
\item[ii] The evaluation map 
\begin{equation*}
\begin{array}{cccc}
\mathcal{L}_{d,g}(E,F) \times E \longrightarrow F, \\
(L,x) \longrightarrow L.x
\end{array}
\end{equation*}
is bilinear continuous.
\item[iii] If $ (G,h) $ is another Fr\'{e}chet space with standard metric, then the composition map 
\begin{equation*} 
\begin{array}{cccc}
\mathcal{L}_{d,g}(F,G) \times \mathcal{L}_{g,h}(E,F) \longrightarrow \mathcal{L}_{d,h}(E,G), \\
(L,H) \longrightarrow L \circ H
\end{array}
\end{equation*}
is bilinear continuous.
\item[iv] The Metric \eqref{metric} is complete and has absolutely convex balls.
\item[v] The group of automorphisms, $ \Aut(E)$, is open in $ \mathcal{L}_d(E)$ with respect to the topology induced by the Metric~\eqref{metric}. And the inversion map
$ \Aut(E) \rightarrow \Aut(E),\, A \rightarrow A^{-1} $ is continuous.
\end{description}
\end{prop}

\begin{prop} \label{open}
Let $ (E,d) $ and $ (F,g) $ be Fr\'{e}chet spaces, and $ g $ a standard metric. The set of isomorphisms of
$ E $ to $ F $, $ \Iso{(E,F)}, $ is open in $ \mathcal{L}_{d,g}(E,F) $ with respect to the topology induced by the Metric \eqref{metric}.
\end{prop}
\begin{proof}
Fix an isomorphism $ i: E \rightarrow F $. Define a map 
\begin{equation*}
\begin{array}{cccc}
i^* : \mathcal{L}_d(E)\longrightarrow \mathcal{L}_{d,g}(E,F), \\
i^* = i \circ k, \; k \in \mathcal{L}_d(E)
\end{array}
\end{equation*}
$ i^* $ is bijective because $ i $ is isomorphism. The composition $ i \circ k $ is 
bilinear continuous by Proposition \ref{p2} $(iii)$. Thus, $ i^* $ is homeomorphism by virtue of the open mapping theorem. 
Since the group of automorphisms of $ E ,\, \Aut ({E}) $, is open in $\mathcal{L}_d(E)$ by Proposition \ref{open} $(v)$
it follows that its image under $ i^* $, which is the set of isomorphisms $ \Iso{(E,F)} $, is open in $ \mathcal{L}_{d,g}(E,F) $. 
\end{proof}
\subsection{Differentiation and $MC^k-$ maps }
\begin{defn}
Let $ E,F $ be Fr\'{e}chet spaces, $ U $ an open subset of $ E $, and $ P:U \rightarrow F $
 a continuous map. Let $CL(E,F)$ be the space of all continuous linear maps from $E$ to $F$ topologized by the compact-open topology. 
 We say $ P $ is differentiable at the point $ p \in U$ if there exists a linear map
$\operatorname{d}P(p): E \rightarrow F$ with $$\operatorname{d}P(p)h =  \lim_{t\rightarrow 0}\tfrac{P(p+th)-P(p)}{t}$$, for all $ h \in E $. 
If $P$ is differentiable at all points $p \in U$, if $\operatorname{d}P(p) : U \rightarrow CL(E,F)$ is continuous for all
$p \in U$ and if the induced map $ P': U \times E \rightarrow F,\,(u,h) \mapsto \operatorname{d}P(u)h $
 is continuous in the product topology, then we say that $ P $ is Keller-differentiable.
  We define $ P^{(k+1)}: U \times E^{k+1} \rightarrow F $ inductively by
\begin{equation*}
P^{(k+1)}(u,f_{1},...,f_{k+1}) = \lim_{t\rightarrow 0}\dfrac{P^{(k)}(u+tf_{k+1})(f_1,...,f_k)- P^{(k)}(u)(f_1,...,f_k)}{t}.
\end{equation*}
\end{defn}
\begin{defn} \label{difr}
If $P$ is Keller-differentiable, $ \operatorname{d}P(p) \in \mathcal{L}_{d,g}(E,F) $ for all $ p \in U $, and the induced map 
$ \operatorname{d}P(p) : U \rightarrow \mathcal{L}_{d,g}(E,F)  $ is continuous, then $ P $ is called b-differentiable. 

We say $ P $ is $ MC^{0} $ and write $ P^0 = P $ if it is continuous. 
We say $P$ is an $ MC^{1} $ and write  $P^{(1)} = P' $ if it is b-differentiable. Let $ \mathcal{L}_{d,g}(E,F)_0 $ be 
the connected component of $ \mathcal{L}_{d,g}(E,F) $ containing the zero map. If $ P $ is b-differentiable and if 
 $V \subseteq U$ is a connected open neighborhood of $x_0 \in U$, then $P'(V)$ is connected and hence contained in the connected component
$P'(x_0) +  \mathcal{L}_{d,g}(E,F)_0 $ of $P'(x_0)$ in $\mathcal{L}_{d,g}(E,F)$. Thus, $P'\mid_V - P'(x_0):V \rightarrow \mathcal{L}_{d,g}(E,F)_0 $
 is again a map between subsets of Fr\'{e}chet spaces. This enables a recursive definition: If $P$ is $MC^1$ and $V$ can be chosen for each
 $x_0 \in U$ such that $P'\mid_V - P'(x_0):V \rightarrow \mathcal{L}_{d,g}(E,F)_0 $ is $ MC^{k-1} $, 
then $ P $ is called an $ MC^k$-map. We make a piecewise definition of $P^{(k)}$ by $ P^{(k)}\mid_V \coloneq \left(P'\mid_V - P'(x_0)\right)^{(k-1)} $
for $x_0$ and $V$ as before. 
The map $ P $ is $ MC^{\infty} $ if it is $ MC^k $ for all $ k \in \mathbb{N}_0 $.
\end{defn}
We shall denote by $\operatorname{D},\operatorname{D^2}$  the first and the second differential, respectively.
We should mention that an appropriate version of the  Chain rule is available and the composition of composable $ MC^k$- maps is again $ MC^k $.
A \textbf{bounded Fr\'{e}chet manifold} is a Hausdorff Second countable topological space with an atlas of coordinate 
charts taking their values in Fr\'{e}chet spaces such that the coordinate transition functions are all $MC^{\infty}$.

Suppose $ M $ is a bounded Fr\'{e}chet manifold. Naturally, we define a bounded Fr\'{e}chet vector bundle 
over $ M $, this is a Fr\'{e}chet vector bundle whose total space is a bounded Fr\'{e}chet manifold. The    
tangent bundle $ TM $ is a bounded Fr\'{e}chet vector bundle over $ M $ whose coordinate transition 
functions are just the tangents $ TP $ of the coordinate transition functions $ P $ for M.

\begin{defn}
A map $ P: M \rightarrow N $ is an $ MC^{k} $ of bounded Fr\'{e}chet manifolds if we can find charts around
 any point in $ M $ and its image in $ N $ such that the local representative of $ P $ in these charts is an
 $ MC^{k}$- map of Fr\'{e}chet spaces. It is called $ MC^{\infty} $ if the local representative in the charts
 is $ MC^{\infty} $.
\end{defn}
For $ k \geq 1 $, an $  MC^k$- map $ P: M \rightarrow N $ of bounded Fr\'{e}chet manifolds induces a tangent
 map $ TP:TM \rightarrow TN $ of their tangent bundles which takes the fibre over $ f \in M $ into the fibre 
over $ P(f) \in N $ and is linear on each fibre. The local representatives for the tangent map $ TP $ are just 
the tangents of the local representatives for $ P $. The \textbf{derivative} of $ P $ at $ f $ is the linear
 map
\begin{equation*}
\operatorname{D}P(f): TM_f  \longrightarrow TN_{P(f)}
\end{equation*} 
induced by $ TP $ on the tangent spaces. When the manifolds are Fr\'{e}chet spaces this agrees with 
Definition ~\ref{difr}.
\section{Lipschitz Fredholm maps and stability}\label{lp}

\begin{defn}
Let $ (E,d) $ and $ (F,g) $ be Fr\'{e}chet spaces, and  $ g $ a standard metric. 
A map $ \varphi \in \mathcal{L}_{d,g}(E,F) $ is called Lipschitz-Fredholm operator if it satisfies 
the following conditions:
\begin{enumerate}
\item The image of $ \varphi $ is closed.
\item The dimension of the kernel of $ \varphi$ is finite.
\item The co-dimension of the image of $ \varphi $ is finite.
\end{enumerate}
\end{defn}
We denote by $ \mathcal{LF}(E,F) $ the set of all Lipschitz-Fredholm operators from $ E $ into
$ F $. For $ \varphi \in \mathcal{LF}(E,F) $ we define the \textbf{index} of $ \varphi $ to be
\begin{equation*}
\Ind \varphi = \dim \ker \varphi - \codim \Img \varphi.
\end{equation*}
A subset $ G $ of a Fr\'{e}chet space $ F $ is called topologically complemented or split in $ F $, if $ F $ is homeomorphic to the topological direct sum $ G \oplus H $, where $ H $ is a subspace of $ F $. We call $ H $ a topological complement of $ G $ in $ F $.
\begin{theorem}[\cite {muller}, Theorem 3.14] \label{comp2}
Let $ F $ be a Fr\'{e}chet space. Then
\begin{description}
\item[i] Every finite-dimensional subspace of $ F $ is closed.
\item[ii] Every closed subspace $ G \subset F $ with $ \codim(G)=\dim (F/G)< \infty $ is topologically complemented in $ F $.
\item[iii] Every finite-dimensional subspace of $ F $ is topologically complemented.
\item [iv]Every linear isomorphism between the direct sum of two closed subspaces and $ F $, $ G \oplus H \rightarrow F $, is a homeomorphism.

\end{description}
\end{theorem}
\begin{theorem}
$ \mathcal{LF}(E,F) $ is open in $\mathcal{L}_{d,g}(E,F)$ with respect to the topology defined by the Metric
 \eqref{metric}. Furthermore, the function $ T \rightarrow \Ind T $ is continuous on $ \mathcal{LF}(E,F) $,
 hence constant on connected components of $ \mathcal{LF}(E,F) $.
\end{theorem}
\begin{proof}
Suppose $\varphi : E \rightarrow F$ is a Lipschitz-Fredholm operator. We have to find a neighbourhood $ N $ of
 $ \varphi $ in $ \mathcal{L}_{d,g}(E,F) $ such that $ N \subset \mathcal{LF}(E,F) $. We can write 
$ E = \ker \varphi \oplus G $, where $ G $ is a topological complement of $ \ker \varphi $ by Theorem 
\ref{comp2} $(iii)$. $ \varphi $ induces linear isomorphism of $ G $ into its image $ \varphi (G) $ 
by virtue of the open mapping theorem, thus we can write $ F = \varphi (G) \oplus H $ for some finite
 dimensional subspace $ H $ in $ F $. The map
\begin{equation*}
\begin{array}{cccc}
\overline{\varphi}:G\oplus H \longrightarrow \varphi(G) \oplus H = F \\
\overline{\varphi}(x,y) \mapsto \varphi .x +y
\end{array}
\end{equation*}
is a linear isomorphism. But the set of linear isomorphisms is open in the space of linear Lipschitz maps, see Proposition ~\ref{open}. Now assume $ \psi \in \mathcal{L}_{d,g}(E,F)$
is in an open neighbourhood of $ \varphi $, say $ N $, it is constructed as follows, consider the space of
linear maps $\mathcal{L}(G\oplus H, F)$. Define the map $\alpha$ 
\begin{equation*}
\begin{array}{cccc}
\alpha : \mathcal{L}_{d,g}(E,F) \rightarrow \mathcal{L}(G\oplus H, F) \\
(\alpha \circ \psi)(x,y)= \psi .x+y
\end{array}
\end{equation*}
Put $N = \alpha^{-1}(\Iso (G\oplus H, F))$.
The map
 \begin{equation*}
\begin{array}{cccc}
\overline{\psi} : G\oplus H \longrightarrow  F \\
\overline{\psi}(x,y) \mapsto \psi .x+y
\end{array}
\end{equation*}
is therefore, a linear isomorphism. It follows that $\dim \ker \psi \leq \dim \ker \varphi $. 
Indeed, consider the projection map $ \gamma : E = G \oplus \ker \varphi \rightarrow \ker \varphi$. 
Since $\overline{\psi}$ is isomorphism it follows that $ G \cap \ker \psi = 0 $. But $ G = \ker \gamma $, 
whence $ \ker (\gamma \vert_{ \ker \psi}) = \ker \gamma \cap \ker \psi = G \cap \ker \psi = 0 $. 
Thus $ \gamma \vert_{ \ker \psi} : \ker \psi \rightarrow  \ker \psi$ is a monomorphism, whence 
$ \dim \ker \psi < \dim \ker \varphi < \infty $. Now let us show that
 $ \codim \Img (\psi) = \dim \frac{F}{\psi(E)} < \infty $ . Let $ K $ be the complement of 
$ \gamma (\ker \psi) $ in $ \ker \varphi $. Then $ \dim K < \infty $, and there exists an isomorphism 
$ E \cong G \oplus K \oplus \ker \psi$. 
Notice that $ \psi \vert _G = \overline{\psi} \vert_{ (G \oplus \, \{ 0 \}) }$ 
and we have natural identifications: 
\begin{equation*}
 \dfrac{F}{\psi(G)} = \dfrac{F}{\overline{\psi}(G \oplus \{ 0 \})} \overset{\iota} \cong H  \;\; and \;\;
 \dfrac{F}{\psi(G \oplus K)} \overset {\kappa}\cong \dfrac{F}{\psi{(E)}}.
\end{equation*}
Then we have the following commutative diagram:
\begin{equation*}
\xymatrix{
\frac{F}{\psi(G)} \ar[d]^{\xi} \ar[r]^{\iota} &H\ar[d]^{\zeta}\\
\frac{F}{\psi(G \oplus K)} \ar[r]^{\kappa}
& \frac{F}{\psi(E)}}
\end{equation*}
Since $ \zeta $ is onto and $ H $ has a finite dimension, we see that 
$ \dim \frac{F}{\psi(E)} < \infty $. So $\psi$ is Lipschitz Fredholm.\\
Evidently,
\begin{equation*}
 \dim \ker \xi = \dim \ker \psi = \dim K  \:\:\: and \:\:\: \codim \Img \psi = \dim H - \dim K 
\end{equation*}
Furthermore, 
\begin{equation*}
\dim \ker \varphi = \dim K + \dim \ker \psi \;\:\: and \:\:\: \codim \Img \varphi = \dim H = \dim K + \dim \frac{F}{\psi(E)}
\end{equation*} 
So we have 
\begin{align*}
\Ind \psi &= \dim \ker \psi - (\dim H - \dim K) \\
&=  \dim \ker \psi + \dim K - \dim H \\
&=  \dim \ker \varphi - \dim H \\
&= \Ind \varphi.
\end{align*}
\end{proof}
In the sequel, we assume that manifolds are connected.
\begin{defn}
Let $ M $ and $ N $ be bounded Fr\'{e}chet manifolds modelled 
on  Fr\'{e}chet spaces $ E $ and $ F $ with standard metrics. 
A Lipschitz-Fredholm map is an $ MC^1$- map $ f:M \rightarrow N $ such that for each $ x \in M $, 
the derivative $ \operatorname{D}P(x): TM_f  \longrightarrow TN_{P(x)} $ is a Lipschitz-Fredholm operator. 
The index of $ f$, denoted by $\Ind{f}$, is defined to be the index of $ \operatorname{D}P(x) $ for some $ x $. Since $ f $ is $ MC^1 $ and $ M $ is connected by Theorem \ref{open} the definition does not depend on the choice of $ x $. 
\end{defn}
\section{Sard's Theorem}\label{sar}
Let $ M $ and $ N $ be bounded Fr\'{e}chet manifolds modelled
 on Fr\'{e}chet spaces $ E $ and $ F $ with standard metrics. Let $ f:M \rightarrow N$ be any $ MC^1 $ map.
 A point $ x \in M $ is called a \textbf{regular point} of $ f $ if 
$ \operatorname{D}P(x): TM_f  \longrightarrow TN_{P(x)} $ is surjective, otherwise is called 
\textbf{critical}. The images of the critical points under $ f $ are called the \textbf{the critical values}
 and their complement the \textbf{regular values}. Note that if $ y \in N $ is not in image of $ f $ it is a 
regular value. We denote the set of critical points of $ f $ by $ \mathcal{C}_f $ and the set of regular values
 by $ \mathcal{R}(f) $ or $ \mathcal{R}_f $. In addition, for $ A \subset M $ we define 
$ \mathcal{R}_ {f \vert A }$ by $ \mathcal{R}_f \vert A = N \setminus f( \mathcal{C}_f \cap A )$. 
In particular, if $ U \subset M $ is open, $ \mathcal{R}_f \vert U = \mathcal{R}(f \vert U) $. 
\begin{theorem}[\cite{glockner}, Proposition 7.1. Inverse Function Theorem for $ MC^k$- maps] \label{invr}
Let $ (E,g) $ be a Fr\'{e}chet space with standard metric $g$. Let $U \subset E$ be open, $x_0 \in U$ and
$ f : U \subset E \rightarrow E $  an $ MC^k$-map, $ k \geq 1 $.
 If $f'(x_0) \in \Aut{(E)}  $, then there exists an open neighborhood
 $ V \subseteq U $ of $ x_0 $ such that $ f(V) $ is open in $ E $ and $ f\vert_V : V \rightarrow f(V) $ 
is an $ MC^k$- diffeomorphism.
\end{theorem}
\begin{theorem}[Local Representation Theorem] \label{rp}
Let $ f: U \subseteq E \rightarrow F $ be an $ MC^k $, $ k \geq 1 $, $ u_0 \in U $ and suppose that 
$ \operatorname{D}f(u_o) $ has closed split image $ F_1 $ with closed topological complement $ F_2 $ 
and split kernel $ E_2 $ with closed topological complement $ E_1 $. Then there are two open sets
 $ U' \subset U \subset E = E_1 \oplus E_2 $ and
$ V \subset F_1 \oplus E_2 $ and an $ MC^k$-diffeomorphism $ \Psi : V \rightarrow U' $,
such that $ (f\circ \Psi)(u,v) = (u,\eta (u,v))$ for all $ (u,v) \in V $, where $ \eta : V \rightarrow E_2 $ 
is an $ MC^k$- map.
\end{theorem}
\begin{proof}
Let $ f = f_1 \times f_2 $, where $ f_i: U \rightarrow F_i,\, i=1,2 $. By virtue of the open mapping theorem 
we have $ \operatorname {D_1}f_1(u_0) = \operatorname f_1(u_0) \vert_{E_1} \in \Iso (E_1, F_1) $.
 Define the map
\begin{equation*}
\begin{array}{cccc}
g : U \subset E_1 \oplus E_2 \rightarrow F_1 \oplus E_2 ,\\
g(u_1, u_2) = (f_1(u_1,u_2),u_2)
\end{array}
\end{equation*}
therefore,
\begin{equation*}
\operatorname{D}g(u).(e_1,e_2)=
\begin{pmatrix}
\operatorname{D_1}f_1(u) & \operatorname {D_2}f_1(u) \\
0 & I_{E_2}
\end{pmatrix}
\begin{pmatrix}
e_1 \\
e_2
\end{pmatrix}
 \end{equation*}
 for all $ u = (u_1,u_2) \in U,\, e_1 \in E_1, \, e_2 \in E_2 $. By hypothesis $ E_2 =
  \ker \operatorname {D}f(u_0) = \ker \operatorname {D}f_1(u_0) $ 
and hence $ \operatorname{D_2}f_1(u_0) = \operatorname{D_2}f_1(u_0) \vert_{E_2} = 0 $. 
Therefore, $ \operatorname {D}g(u_0) \in \Iso{(E, F_1 \oplus E_2)} $. \\
By the inverse function theorem, there are open sets $ U' $ and $ V $ and an $ MC^k$- 
diffeomorphism $ \Psi : V \rightarrow U' $ such that
 $ u_0 \in U' \subset U \subset E, \, g(u_0) \in V \subset F_1 \oplus E_2$,
 and $ \Psi ^{-1} = g \vert_{ U^{'}} $. Hence if $ (u,v) \in V $, then
 $ (u,v) = (g \circ \Psi)(u,v) = g(\Psi_1(u,v),\Psi_2(u,v))= (f_1 \circ \Psi_1(u,v),\Psi_2(u,v)),  $
 where $ \Psi = \Psi_1 \times \Psi_2 $. This shows that $ \Psi_2(v,v)=v $ and $ (f_1 \circ \Psi)(u,v) = u $.
 Define $ \eta = f_2 \circ \Psi $, then 
$ (f \circ \Psi)(u,v) = (f_1 \circ \Psi (u,v),f_2 \circ \Psi (u,v))=(u, \eta (u,v)) $.
\end{proof}
A map $ f $ between topological spaces is called \textbf{locally closed} if for any point $ x $ in 
the domain of $ f $ there exists an open neighborhood $ U $ such that 
$ f \vert_ {\overline{U}} $ is a closed map. 
\begin{lemma}
Let $ f : E \rightarrow F $ be a Lipschitz-Fredholm map between Fr\'{e}chet spaces with standard metrics. 
Then $ f $ is locally closed. 
\end{lemma}
\begin{proof}
Since $ f $ is Fredholm it has split image $ F_1 $ with topological complement $ F_2 $ and split kernel
 $ E_2 $ with topological complement $ E_1 $. By the local representation theorem there are two open sets
 $ U \subset E_1 \oplus E_2 $ and $ V \subset F_1 \oplus E_2  $ and an $ MC^k$- diffeomorphism
 $ \Psi : V \rightarrow U $ such that $ (f \circ \Psi)(u,v) = (u, \eta(u,v)) $ for all $ (u,v) \in V $, 
where $ \eta : V \rightarrow E_2  $ is an $ MC^k -$ map. Suppose $ U_1 \subset F_1 $ and $ U_2 \subset E_2 $ 
are open subsets and $\overline{ U_2} $ is compact. Let $ U' = U_1 \times U_2 \subset U $ so that 
$ \overline{U'} = \overline{U}_1 \times \overline{U}_2 $ and $ \overline{U'} \subset \overline{U}$. 
Suppose $ A \subset \overline{U'} $ is closed and a sequence
 $ \lbrace (y_i, z_i) = (y_i,\eta (y_i,x_i)) \rbrace \subset f(A) $ converges to 
$ (y,z) $, where $ \lbrace (y_i,x_i) \rbrace $ is a sequence in $ A $, we need to show that
 $ (y,z) \in f(A)$. By assumption we have $ \lbrace x_i \rbrace \subset \overline{U}_2 $, and since
 $ \overline{U}_2 $ is a compact subset of a finite dimensional topological vector space, we may assume
 $ x_i \rightarrow x \in \overline{U}_2 $. Then $ (y_i,x_i) \rightarrow (y,x) $. Since $ A $ is 
closed then $ (y,x) \in A $. By continuity of $ f $ we see that $ \lbrace f(y_i,x_i)=(y_i,z_i) \rbrace $ 
converges to $ f(y,x) $, but $ f(y,x) \in f(A) $ thus $ f(A) $ is closed.
\end{proof}
\begin{theorem}
Let $ M $ and $ N $ be bounded Fr\'{e}chet manifolds  modelled 
on Fr\'{e}chet spaces $ E$ and $F $ with standard metrics. Let $ f: M \rightarrow N $ be an 
$ MC^k$- Lipschitz-Fredholm map with $  k > \max \lbrace {\Ind f,0} \rbrace $. Then the set of regular
values of $ f $ is residual in $ N $.
\end{theorem}
\begin{proof}
It is enough to verify that  every $ m  \in M$ has a neighborhood $ Z $ such that
 $ \mathcal{R} (f \vert { \bar Z}) $ is open and dense in $ N $. Then since $ M $ is second countable
we can find a countable open cover $\lbrace Z_i \rbrace$ of $ M $ with $ \mathcal{R}_f\vert \bar{Z_i}$ 
open and dense in $N$. Since $ \mathcal{R}_f = \bigcap _i \mathcal{R}_f \vert{\bar{Z_i}} $, it will follows
 $ \mathcal{R}_f $ is residual.\\
Choose $ m \in M $, we will construct a neighborhood $ Z $ of $ m $ so that
$ \mathcal{R}({f \vert { \bar Z}}) $ is open and dense.
By the local representation theorem we may find charts $ (U,\phi) $ at $ m $ and $( V, \psi) $ at $ f(m) $
 such that $ \phi (U) \subset E \times \mathbb{R}^n,\,\psi(V) \subset E \times \mathbb{R}^p $,  and local representative 
$ f_{\phi\psi}= \psi \circ f \circ \psi ^{-1} $ of $ f $ has the form $f_{\phi\psi}(e,x) = (e, \eta (e,x)) $ 
for $ (e,x) \in \phi(U) $, where $ E $ is a Fr\'{e}chet space, $ x \in \mathbb{R} ^n, e \in E,$ and
 $ \eta : \phi(U) \rightarrow \mathbb{R}^p $. The index of $ T_mf $ is $ n-p $ and so 
$ k > \max \lbrace n-p, 0 \rbrace $.\\
To show that $ \mathcal{R}({f \vert U}) $ is dense in $ N $, it is enough to show that 
$ \mathcal{R}(f_{\phi\psi}) $ is dense in $ E \times \mathbb{R}^p $. For $ e \in E,\,(e,x) \in \psi(U) $,
 define $ \eta_e(x) = \eta(e,x) $. For each $ e, \eta_e $ is a $ C^k$- map defined on an open set of 
$ \mathbb{R}^n $, then by Sard's theorem, $ \mathcal{R}(\eta_e )$ is dense in $ \mathbb{R}^p $ for each $ e \in E $.
 But for $ (e,x) \in \psi(U) $, we have 
\begin{equation*}
\operatorname{D}f_{\phi \psi}(e,x) =
\begin{pmatrix}
I & 0 \\ * & \operatorname{D}\eta_e(x)
\end{pmatrix}
\end{equation*}
So $ \operatorname{D}f_{\phi \psi}(e,x) $ is surjective if and only if $\operatorname{D}\eta_e(x)  $ 
is surjective, hence for $ e \in E $
\begin{equation*}
\lbrace e \rbrace \times \mathcal{R}({\eta}_e) = \mathcal{R}(f _{\phi \psi}) \cap (\lbrace e\rbrace \times \mathbb{R}^p)
\end{equation*}
and so $ \mathcal{R}(f_{\phi \psi}) $ intersects every plane $\lbrace e\rbrace \times \mathbb{R}^p  $ in a dense set 
and is, therefore, dense in $ E \times \mathbb{R}^p $. So $ \mathcal{R}({f \vert U}) $ is dense.\\
Since $ f $ is locally closed we can chose an open neighborhood $ Z $ of $ m $ such that
 $ \bar{Z} \subset U$ and $ f \vert_{\bar Z} $ is closed. Since  $ \mathcal{C}_f $ is closed in $ M $,
 then $ f (\bar{Z} \cap \mathcal{C}_f) $ is closed in $ N $, and so
$ \mathcal{R}({f \vert {\bar{Z}}}) = N \setminus f(\bar Z \cap C_f)$ is open in $ N $. Since
$ \mathcal{R}({f \vert U}) \subset \mathcal{R}({f\vert{\bar{Z}}})  $ then $  \mathcal{R}({f\vert{\bar{Z}}}) $
is dense as well. This complete the proof.
\end{proof}

\end{document}